\documentclass[twocolumn,amsthm]{autart}    

\usepackage{graphicx}          

\usepackage{enumerate}

\theoremstyle{plain}
\newtheorem{theorem}{Theorem}
\newtheorem{lemma}{Lemma}
\newtheorem{corollary}{Corollary}
\newtheorem{proposition}{Proposition}

\theoremstyle{remark}
\newtheorem{remark}{Remark}
\newtheorem{definition}{Definition}
\newtheorem{example}{Example}

\usepackage{amsmath}
\usepackage{amssymb}

\begin{document}

\begin{frontmatter}

\title{Provably Safe and Robust Learning-Based Model Predictive Control\thanksref{footnoteinfo}} 

\thanks[footnoteinfo]{Corresponding author A.~Aswani.}

\author[UCB]{Anil Aswani}\ead{aaswani@eecs.berkeley.edu},    
\author[UCB]{Humberto Gonzalez}\ead{hgonzale@eecs.berkeley.edu},               
\author[UCB]{S. Shankar Sastry}\ead{sastry@eecs.berkeley.edu},   
\author[UCB]{Claire Tomlin}\ead{tomlin@eecs.berkeley.edu}  

\address[UCB]{Electrical Engineering and Computer Sciences, Berkeley, CA 94720}  
          

\begin{abstract}                          
Controller design faces a trade-off between robustness and performance, and the reliability of linear controllers has caused many practitioners to focus on the former.  However, there is renewed interest in improving system performance to deal with growing energy constraints.  This paper describes a learning-based model predictive control (LBMPC) scheme that provides deterministic guarantees on robustness, while statistical identification tools are used to identify richer models of the system in order to improve performance; the benefits of this framework are that it handles state and input constraints, optimizes system performance with respect to a cost function, and can be designed to use a wide variety of parametric or nonparametric statistical tools.  The main insight of LBMPC is that safety and performance can be decoupled under reasonable conditions in an optimization framework by maintaining two models of the system.  The first is an approximate model with bounds on its uncertainty, and the second model is updated by statistical methods.  LBMPC improves performance by choosing inputs that minimize a cost subject to the learned dynamics, and it ensures safety and robustness by checking whether these same inputs keep the approximate model stable when it is subject to uncertainty.  Furthermore, we show that if the system is sufficiently excited, then the LBMPC control action probabilistically converges to that of an MPC computed using the true dynamics.
\end{abstract}

\end{frontmatter}

\section{Introduction}

Tools from control theory face an inherent trade-off between robustness and performance.  Stability can be derived using approximate models, but optimality requires accurate models.  This has driven research in adaptive \cite{sastry1989,sastryisidori1989,narendra1990,astrom1995,rovithakis1994} and learning-based \cite{xu2003,anderson2007,tedrake2009,abbeel2010,ljung2011} control.  Adaptive control reduces conservatism by modifying controller parameters based on system measurements, and learning-based control improves performance by using measurements to refine models of the system.  However, learning by itself cannot ensure the properties that are important to controller safety and stability \cite{barto2004,aswani2009,aswani2010}.

The motivation of this paper is to design a control scheme than can (a) handle state and input constraints, (b) optimize system performance with respect to a cost function, (c) use statistical identification tools to learn model uncertainties, and (d) provably converge.  The main challenge is combining (a) and (c): Statistical methods converge in a probabilistic sense, and this is not strong enough for the purpose of providing deterministic guarantees of safety.  Showing (d) is also difficult because of the differences between statistical and dynamical convergence.  

We introduce a form of robust, adaptive model predictive control (MPC) that we refer to as learning-based model predictive control (LBMPC).  The main insight of LBMPC is that performance and safety can be decoupled in an MPC framework by using reachability tools \cite{asarin2000,aswani07a,rakovic2006,chutinan99,asarin2003,stursberg2003,mitchell2005}.  In particular, LBMPC improves performance by choosing inputs that minimize a cost subject to the dynamics of a learned model that is updated using statistics, while ensuring safety and stability by using theory from robust MPC \cite{borelli2009,chisci2001,langson2004,limon2010} to check whether these same inputs keep a nominal model stable when it is subject to uncertainty.

LBMPC is similar to other variants of MPC.  For instance, linear parameter-varying MPC (LPV-MPC) has a model that changes using successive online linearizations of a nonlinear model \cite{kothare2000,falcone2008}; the difference is that LBMPC updates the models using statistical methods, provides robustness to poor model updates, and can involve nonlinear models.  Other forms of robust, adaptive MPC \cite{Fukushima2007301,adetola2011} use an adaptive model with an uncertainty measure to ensure robustness, while LBMPC uses a learned model to improve performance and a nominal model with an uncertainty measure to provide robustness.

Here, we focus on LBMPC for when the nominal model is linear and has a known level of uncertainty.  After reviewing notation and definitions, we formally define the LBMPC optimization problem.  Deterministic theorems about safety, stability, and robustness are proved.  Next, we discuss how learning is incorporated into the LBMPC framework using parametric or nonparametric statistical tools.  Provided sufficient excitation of the system, we show convergence of the control law of LBMPC to that of an MPC that knows the true dynamics.  The paper concludes by discussing applications of LBMPC to three experimental testbeds \cite{aswani_proc,aswani_quad_2011,bouffard2011,aswani2012nmpc} and to a simulated jet engine compression system \cite{mooregreitzer1986,epstein1989,kristic1995}.

\section{Preliminaries}

In this section, we define the notation, the model, and summarize three results on estimation and filtering.  Note that polytopes are assumed to be convex and compact.

\subsection{Mathematical Notation}
We use $A'$ to denote the transpose of $A$, and subscripts denote time indices.  Marks above a variable distinguish the state, output, and input of different models of the same system.  For instance, the true system has state $x$, the linear model with disturbance has state $\overline{x}$, and the model with oracle has state $\tilde{x}$.

A function $\gamma : \mathbb{R}_{+} \rightarrow \mathbb{R}_{+}$ is type-$\mathcal{K}$ if it is continuous, strictly increasing, and $\gamma(0) = 0$ \cite{sastry1999}.  Function $\beta : \mathbb{R}_+ \times \mathbb{R}_+ \rightarrow \mathbb{R}_+$ is type-$\mathcal{KL}$ if for each fixed $t \geq 0$, the function $\beta(\cdot,t)$ is type-$\mathcal{K}$, and for each fixed $s \geq 0$, the function $\beta(s,\cdot)$ is decreasing and $\beta(s,t)\rightarrow 0$ as $t\rightarrow \infty$ \cite{jiangwang2001}.  Also, $V_m(x)$ is a Lyapunov function for a discrete time system if (a) $V_m(\overline{x}_s) = 0$ and $V_m(x) > 0, \forall x \neq \overline{x}_s$; (b) $\alpha_1(\|x-\overline{x}_s\|) \leq V_m(x) \leq \alpha_2(\|x-\overline{x}_s\|)$, where $\alpha_1,\alpha_2$ are type-$\mathcal{K}$ functions; (c) $\overline{x}_s$ lies in this interior of the domain of $V_m(x)$; and (d) $V_{m+1}(x_{m+1})-V_m(x_m) < 0$ for states $x_m \neq \overline{x}_s$ of a dynamical system.

Let $\mathcal{U},\mathcal{V},\mathcal{W}$ be sets.  Their Minkowski sum \cite{schneider1993} is $\mathcal{U} \oplus \mathcal{V} = \{u + v : u \in \mathcal{U}; v \in \mathcal{V}\}$, and their Pontryagin set difference \cite{schneider1993} is $\mathcal{U} \ominus \mathcal{V} = \{u : u \oplus \mathcal{V} \subseteq \mathcal{U}\}$.  This set difference is not symmetric, and so the order of operations is important; also, the set difference can result in an empty set.  The linear transformation of $\mathcal{U}$ by matrix $T$ is given by $T\mathcal{U} = \{Tu : u \in \mathcal{U}\}$.  Some useful properties \cite{schneider1993,kolmanovsky1998} include: $(\mathcal{U}\ominus\mathcal{V})\oplus\mathcal{V}\subseteq\mathcal{U}$, $(\mathcal{U} \ominus (\mathcal{V} \oplus \mathcal{W})) \oplus \mathcal{W} \subseteq \mathcal{U} \ominus \mathcal{V}$, $(\mathcal{U}\ominus\mathcal{V})\ominus\mathcal{W} \subseteq\mathcal{U}\ominus(\mathcal{V}\oplus\mathcal{W})$, and $T(\mathcal{U}\ominus\mathcal{V})\subseteq T\mathcal{U}\ominus T\mathcal{V}$.

For a sequence $f_n$ and rate $r_n$, the notation $f_n = O(r_n)$ means that $\exists M,N > 0$ such that $\|f_n\| \leq M\|r_n\|$, for all $n > N$.  For a random variable $f_n$, constant $f$, and rate $r_n$, the notation $\|f_n-f\| = O_p(r_n)$ means that given $\epsilon > 0$, $\exists M,N > 0$ such that $\mathbb{P}(\|f_n - f\|/r_n > M) < \epsilon$, for all $n > N$.  The notation $f_n \xrightarrow{p} f$ means that there exists $r_n \rightarrow 0$ such that $\|f_n-f\| = O_p(r_n)$.

\subsection{Model}

Let $x \in \mathbb{R}^p$ be the state vector, $u \in \mathbb{R}^m$ be the control input, and $y \in \mathbb{R}^q$ be the output.  We assume that the states $x \in \mathcal{X}$ and control inputs $u \in \mathcal{U}$ are constrained by the polytopes $\mathcal{X},\mathcal{U}$.  The true system dynamics are
\begin{equation}
\label{eqn:smodel}
x_{n+1} = Ax_n + Bu_n + g(x_n,u_n)
\end{equation}
and $y_n = Cx_n$, where $A,B,C$ are matrices of appropriate size and $g(x,u)$ describes the unmodeled (possibly nonlinear) dynamics.  The intuition is that we have a nominal linear model with modeling error.  The term uncertainty is used interchangeably with modeling error.  

We assume that the modeling error $g(x,u)$ of (\ref{eqn:smodel}) is bounded and lies within a polytope $\mathcal{W}$, meaning that $g(x,u) \in \mathcal{W}$ for all $(x,u) \in (\mathcal{X},\mathcal{U})$.  This assumption is not restrictive in practice because it holds whenever $g(x,u)$ is continuous, since $\mathcal{X},\mathcal{U}$ are bounded.  Moreover, the set $\mathcal{W}$ can be determined using techniques from uncertainty quantification \cite{biegler2011large}; for example, the residual error from model fitting can be used to compute this uncertainty.

\subsection{Estimation and Filtering}

Simultaneously performing state estimation and learning unmodeled dynamics requires measuring all states \cite{aswani2011_techrep}, except in special cases \cite{aswani_quad_2011}.  We focus on the case in which all states are measured (i.e, $C = \mathbb{I}$).  It is possible to relax these assumptions by using set theoretic estimation methods (e.g., \cite{milanese1982}), but we do not consider those extensions here.  For simplicity of presentation, we assume that there is no measurement noise; however, our results extend to the case with measurement noise by simply replacing the modeling error $\mathcal{W}$ in our results with $\mathcal{W} \oplus \mathcal{D}$, where $\mathcal{D}$ is a polytope encapsulating the effect of bounded measurement noise.

\section{Learning-Based MPC}
\label{section:lbmpc}

This section presents the LBMPC technique.  The first step is to use reachability tools to construct a terminal set with robustness properties for the LBMPC, and this terminal set is important for proving the stability, safety, and robustness properties of LBMPC.  The terminal constraint set is typically used to guarantee both feasibility and convergence \cite{mayne2000}.  We decouple performance from robustness by identifying feasibility with robustness and convergence with performance.

One novelty of LBMPC is that different models of the system are maintained by the controller.  In order to delineate the variables of the various models, we add marks above $x$ and $u$.  The true system (\ref{eqn:smodel}) has state $x$ and input $u$.  The nominal linear model with uncertainty has state $\overline{x}$ and input $\overline{u}$; its dynamics are given by
\begin{equation}
\label{eqn:linear}
\overline{x}_{n+1} = A\overline{x}_n + B\overline{u}_n + d_n,
\end{equation}
where $d_n \in \mathcal{W}$ is a disturbance.  Because $g(x,u) \in \mathcal{W}$, the $d_n$ reflects the uncertain nature of modeling error.

For the learned model, we denote the state $\tilde{x}$ and input $\tilde{u}$.  Its dynamics are $\tilde{x}_{n+1} = A\tilde{x}_n + B\tilde{u}_n + \mathcal{O}_n(\tilde{x}_n,\tilde{u}_n)$, where $\mathcal{O}_n$ is a time-varying function that is called the oracle.  The reason we call this function the oracle is in reference to computer science in which an oracle is a black box that takes in inputs and gives an answer: LBMPC only needs to know the value (and gradient when doing numerical computations) of this function at a finite set of points; and yet, the mathematical structure and details of how the oracle is computed are not relevant to the stability and robustness properties of LBMPC.


\subsection{Construction of an Invariant Set}
\label{sect:fspt}

We begin by recalling two facts \cite{limon2010}.  First, if $(A,B)$ is stabilizable, then the set of steady-state points are $\overline{x}_s = \Lambda\theta$ and $\overline{u}_s = \Psi\theta$, where $\theta \in \mathbb{R}^m$ and $\Lambda,\Psi$ are full column-rank matrices with suitable dimensions.  These matrices can be computed with a null space computation, by noting that $\text{range}([\Lambda'\ \Psi']') = \text{null}([(\mathbb{I}-A)\ -B])$.  Second, if $(A+BK)$ is Schur stable (i.e., all eigenvalues have magnitude strictly less than one), then the control input $\overline{u}_n = K(\overline{x}_n - \overline{x}_s) + \overline{u}_s= K\overline{x}_n + (\Psi - K\Lambda)\theta$ steers (\ref{eqn:linear}) to steady-state $\overline{x}_s = \Lambda\theta$ and $\overline{u}_s = \Psi\theta$, whenever $d_n \equiv 0$.

These facts are useful because they can be used to construct a robust reachable set that serves as the terminal constraint set for LBMPC.  The particular type of reach set we use is known as a maximal output admissible disturbance invariant set $\Omega \subseteq \mathcal{X} \times \mathbb{R}^m$.  It is a set of points such that any trajectory of the system with initial condition chosen from this set and with control $\overline{u}_n$ remains within the set for any sequence of bounded disturbance, while satisfying constraints on the state and input \cite{kolmanovsky1998}.  

These properties of $\Omega$ are formalized as (a) constraint satisfaction: 
\begin{multline}
\label{eqn:consat}
\Omega \subseteq \{(\overline{x},\theta) : \overline{x} \in \mathcal{X};\Lambda\theta \in \mathcal{X}; \\
K\overline{x} + (\Psi - K\Lambda)\theta\in\mathcal{U};\Psi\theta \in \mathcal{U}\},
\end{multline}
and (b) disturbance invariance:
\begin{equation}
\label{eqn:distinv}
\begin{bmatrix} A+BK & B(\Psi-K\Lambda) \\ 0 & \mathbb{I} \end{bmatrix}\Omega \oplus (\mathcal{W} \times \{0\}) \subseteq \Omega.
\end{equation}
Recall that the $\theta$ component of the set is a parametrization of which points can be tracked using control $\overline{u}_n$.  


The set $\Omega$ has an infinite number of constraints in general, though arbitrarily good approximations can be computed in a finite number of steps \cite{kolmanovsky1998,limon2010,rakovic2010}.  These approximations maintain both disturbance invariance and constraint satisfaction, and these are the properties which are used in the proofs for our MPC scheme.  So even though our results are for $\Omega$, they equally hold true for appropriately computed approximations.  
%
%
%

\subsection{Stability and Safety of LBMPC}

LBMPC uses techniques from a type of robust MPC known as tube MPC \cite{chisci2001,langson2004,limon2010}, and it enlarges the feasible domain of the control by using tracking ideas from \cite{chisci2003,limon2008}.  The idea of tube MPC is that given a nominal trajectory of the linear system (\ref{eqn:linear}) without disturbance, then the trajectory of the true system (\ref{eqn:smodel}) is guaranteed to lie within a tube that surrounds the nominal trajectory.  A linear feedback $K$ is used to control how wide this tube can grow.  Moreover, LBMPC fixes the initial condition of the nominal trajectory as in \cite{chisci2001,langson2004}, as opposed to letting the initial condition be an optimization variable as in \cite{limon2010}.

Let $N$ be the number of time steps for the horizon of the MPC.  The width of the tube at the $i$-th step, for $i \in \mathcal{I} = \{0,\ldots,N-1\}$, is given by a set $\mathcal{R}_i$, and the constraints $\mathcal{X}$ are shrunk by the width of this tube.  The result is that if the nominal trajectory lies in $\mathcal{X}\ominus\mathcal{R}_i$, then the true trajectory lies in $\mathcal{X}$.  Similarly, suppose that the $N$-th step of the nominal trajectory lies in $\text{Proj}_x(\Omega) \ominus\mathcal{R}_N$, where $\text{Proj}_x(\Omega) = \Omega_x = \{x : \exists \theta \text{ s.t. } (x,\theta) \in \Omega\}$; then the true trajectory lies in $\text{Proj}_x(\Omega)$, and the invariance properties of $\Omega$ imply that there exists a control that keeps the system stable even under disturbances.

The following optimization problem defines LBMPC
\begin{align}
&V_n(x_n) = \textstyle\min_{c,\theta} \psi_n(\theta,\tilde{x}_n,\ldots,\tilde{x}_{n+N}, \nonumber\\
&\qquad \qquad \qquad \qquad \qquad \qquad \qquad \check{u}_n,\ldots,\check{u}_{n+N-1}) \label{eqn:glmpc}\\
&\text{subject to: } \nonumber \\
&\qquad\left.\begin{aligned}\tilde{x}_n = x_n, \quad\overline{x}_n = x_n\end{aligned}\right. \label{eqn:ic} \\
&\qquad\left.\begin{aligned}\tilde{x}_{n+i+1} = A\tilde{x}_{n+i} + B\check{u}_{n+i} + \mathcal{O}_n(\tilde{x}_{n+i},\check{u}_{n+i})\end{aligned}\right. \label{eqn:nle}\\ 
&\qquad\left.\begin{aligned}
&\overline{x}_{n+i+1} = A\overline{x}_{n+i} + B\check{u}_{n+i}\\
&\check{u}_{n+i} = K\overline{x}_{n+i} + c_{n+i}\\
&\overline{x}_{n+i+1} \in \mathcal{X}\ominus \mathcal{R}_i,\quad \check{u}_{n+i} \in \mathcal{U}\ominus K\mathcal{R}_i\\
&(\overline{x}_{n+N},\theta) \in \Omega \ominus (\mathcal{R}_N \times \{0\})
\end{aligned}\qquad\right\}\label{eqn:lc}
\end{align}
for all $i \in \mathcal{I}$ in the constraints; $K$ is the feedback gain used to compute $\Omega$; $\mathcal{R}_0 = \{0\}$ and $\mathcal{R}_i = \bigoplus_{j = 0}^{i-1}(A+BK)^j\mathcal{W}$; $\mathcal{O}_n$ is the oracle; and $\psi_n$ are non-negative functions that are Lipschitz continuous in their arguments.  Note that the Lipschitz assumption is not restrictive because it is satisfied by costs with bounded derivatives; for example, linear and quadratic costs satisfy this due to the boundedness of states and inputs.  Also note that the same control $\check{u}[\cdot]$ is applied to both the nominal and learned models.

\begin{remark}
The cost $\psi_n$ is a function of the states of the learned model, which uses the oracle to update the nominal model.  The cost function may contain a terminal cost, an offset cost, a stage cost, etc.  An interesting feature of LBMPC is that its stability and robustness properties do not depend on the actual terms within the cost function; this is one of the reasons that we state that LBMPC decouples safety (i.e., stability and robustness) from performance (i.e., having the cost be a function of the learned model).
\end{remark}

\begin{remark}
The constraints in (\ref{eqn:lc}) are taken from \cite{chisci2001} and are robustly imposed on the nominal linear model (\ref{eqn:linear}), taking into account the prior bounds on the unmodeled dynamics of the nominal model $g(x,u)$.  The reason that the constraints are not relaxed to exploit the refined results of the oracle (as in \cite{Fukushima2007301,adetola2011}) is that this provides robustness to the situation in which the learned model is not a good representation of the true dynamics.  It is known that the performance of a learning-based controller can be arbitrarily bad if the learned model does not exactly match the true model \cite{barto2004}; imposing the constraints on the nominal model, instead of the learned model, protects against this situation.
\end{remark}

\begin{remark}
There is another, more subtle reason for maintaining two models.  Suppose that the oracle is bounded by a polytope $\mathcal{O}_n \in \mathcal{P}$, where $\mathcal{P}$ is a polytope; then, the worst case error between the true model (\ref{eqn:smodel}) and the learned model (\ref{eqn:nle}) lies within the polytope $\mathcal{W} \oplus \mathcal{P}$, which is strictly larger than $\mathcal{W}$ whenever $\mathcal{P} \neq \{0\}$.  Intuitively, this means that if we were to use the worst-case bounded learned model in the constraints, then the constraints will be reduced by a larger amount $\mathcal{W} \oplus \mathcal{P}$; this is in contrast to using the nominal model in which case the constraints are reduced by only $\mathcal{W}$.
\end{remark}

Note that the value function $V_n(x_n)$ (i.e., the value of the objective (\ref{eqn:glmpc}) at its minimum), the cost function $\psi_n$, and the oracle $\mathcal{O}_n$ can be time-varying because they are functions of $n$.  It is important that the oracle be allowed to be time-varying, because it is updated using statistical methods as time advances and more data is gathered.  This is discussed in more detail in the next section.

Let $\mathcal{M}_n$ be a feasible point for the LBMPC scheme (\ref{eqn:glmpc}) with initial state $x_n$, and denote a minimizing point of (\ref{eqn:glmpc}) as $\mathcal{M}_n^*$. The states and inputs predicted by the linear model (\ref{eqn:linear}) for point $\mathcal{M}_n$ are denoted $\overline{x}_{n+i}[\mathcal{M}_n]$ and $\overline{u}_{m+n}[\mathcal{M}_n]$, for $i \in I$.  In this notation, the control law is explicitly given by 
\begin{equation}
\label{eqn:feedbackmpc}
u_m[\mathcal{M}_n^*] = Kx_n + c_n[\mathcal{M}_n^*].
\end{equation}
This MPC scheme is endowed with robust feasibility and constraint satisfaction properties, which in turn imply stability of the closed-loop control provided by LBMPC.  The equivalence between these properties and stability holds because of the compactness of constraints $\mathcal{X},\mathcal{U}$.

\begin{theorem}
\label{theorem:robustmpc}
If $\Omega$ has the properties defined in Sect.~\ref{sect:fspt} and $\mathcal{M}_n = \{c_n,\ldots,c_{n+N-1},\theta_n\}$ is feasible for the LBMPC scheme (\ref{eqn:glmpc}) with $x_n$, then applying the control (\ref{eqn:feedbackmpc}) gives
\renewcommand{\labelenumi}{\alph{enumi})}
\begin{enumerate}

\item Robust feasibility: there exists a feasible $\mathcal{M}_{n+1}$ for $x_{n+1}$;

\item Robust constraint satisfaction: $x_{n+1} \in \mathcal{X}$.
\end{enumerate}

\end{theorem}

\begin{proof}
The proof follows a similar line of reasoning as Lemma 7 of \cite{chisci2001}.  We begin by showing that the following point $\mathcal{M}_{n+1} = \{c_{n+1},\ldots,c_{n+N-1},0,\theta_n\}$ is feasible for $x_{n+1}$; the results follow as consequences of this.

Let $d_{n+1+i}[\mathcal{M}_n] = (A+BK)^i g(x_n,u_n)$, and note that $d_{n+1+i}[\mathcal{M}_n] \in (A+BK)^{i+1}\mathcal{W}$.  Some algebra gives the predicted states for $i = 0,\ldots,N-1$ as $\overline{x}_{n+1+i}[\mathcal{M}_{n+1}] = \overline{x}_{n+1+i}[\mathcal{M}_n] + d_{n+1+i}[\mathcal{M}_n] \label{eqn:mp1}$ and predicted inputs for $i = 0,\ldots,N-2$ as $\check{u}_{n+1+i}[\mathcal{M}_{n+1}] = \check{u}_{n+1+i}[\mathcal{M}_n] + Kd_{n+1+i}[\mathcal{M}_n]$.

Because $\mathcal{M}_n$ is feasible, this means by definition that $\overline{x}_{n+1+i}[\mathcal{M}_n] \in \mathcal{X} \ominus \mathcal{R}_{i+1}$ for $i = 0,\ldots,N-1$.  Combining terms gives $\overline{x}_{n+1+i}[\mathcal{M}_{n+1}] \in \mathcal{X} \ominus (\mathcal{R}_i \oplus (A+BK)^{i+1}\mathcal{W}) \oplus (A+BK)^{i+1}\mathcal{W}$.  It follows that $\overline{x}_{n+1+i}[\mathcal{M}_{n+1}] \in \mathcal{X} \ominus \mathcal{R}_i$ for $i = 0,\ldots,N-1$.  Similar reasoning gives that $\check{u}_{n+1+i}[\mathcal{M}_{n+1}] \in \mathcal{U} \ominus K\mathcal{R}_i$ for $i = 0,\ldots,N-2$.

The same argument gives $(\overline{x}_{n+1+N-1}[\mathcal{M}_{n+1}],\theta_n) \in \Omega \ominus (\mathcal{R}_{N-1} \times \{0\}) \subset \Omega.$  Now by construction of $\mathcal{M}_{n+1}$, it holds that $\check{u}_{n+1+N-1}[\mathcal{M}_{n+1}] = Mp$, where $M = [K\ (\Psi-K\Lambda)]$ is a matrix and $p = (\overline{x}_{n+1+N-1}[\mathcal{M}_{n+1}],\theta_n)$ is a point.  Therefore, we have $\check{u}_{n+1+N-1}[\mathcal{M}_{n+1}] = Mp \subseteq M\Omega \ominus M(\mathcal{R}_{N-1}\times\{0\}) = M\Omega \ominus K\mathcal{R}_{N-1}$.  However, the constraint satisfaction property of $\Omega$ (\ref{eqn:consat}) implies that $M\Omega \subseteq \mathcal{U}$.  Consequently, we have that $\check{u}_{n+1+N-1}[\mathcal{M}_{n+1}] \in \mathcal{U} \ominus K\mathcal{R}_{N-1}$.

Next, observe that the control $\check{u}_{n+1+N-1}[\mathcal{M}_{n+1}]$ leads to $\overline{x}_{n+1+N}[\mathcal{M}_{n+1}] = ([A\ 0] + BM)p$.  Consequently, we have $\overline{x}_{n+1+N}[\mathcal{M}_{n+1}] \in ([A\ 0] + BM)\Omega \ominus (A+BK)\mathcal{R}_{N-1}$.  As a result of the disturbance invariance property of $\Omega$ (\ref{eqn:distinv}), it must be that $(\overline{x}_{n+1+N}[\mathcal{M}_{n+1}], \theta_n) \in (\Omega \ominus (\mathcal{W}\times\{0\})) \ominus ((A+BK)\mathcal{R}_{N-1}\times\{0\}) = \Omega\ominus(\mathcal{R}_N\times\{0\})$.  This completes the proof for part (a).

Similar arithmetic shows that the true, next state is $x_{n+1}[\mathcal{M}_n] = \overline{x}_{n+1}[\mathcal{M}_n] + w_n$ where $w_n = g(x_n,u_n) \in \mathcal{W}$.  Since $\mathcal{M}_n$ is a feasible point, it holds that $\overline{x}_{n+1}[\mathcal{M}_n] \in \mathcal{X} \ominus \mathcal{W}$. This implies that $x_{n+1}[\mathcal{M}_n] = \overline{x}_{n+1}[\mathcal{M}_n] + w_n \in (\mathcal{X} \ominus \mathcal{W}) \oplus \mathcal{W} \subseteq \mathcal{X}$; this proves part (b).
\end{proof}

\begin{corollary}
If $\Omega$ has the properties defined in Sect. \ref{sect:fspt} and $\mathcal{M}_0$ is feasible for the LBMPC scheme (\ref{eqn:glmpc}) with initial state $x_0$, then the closed-loop system provided by LBMPC is (a) stable, (b) satisfies all state and input constraints, and (c) feasible, for all points of time $n \geq 0$.
\end{corollary}

\begin{remark}
Robust feasability and constraint satisfaction, as in Theorem \ref{theorem:robustmpc}, trivially imply this result.
\end{remark}

\begin{remark}
These results apply to the case where $\psi_n,\mathcal{O}_n$ are time-varying; this allows, for example, changing the set point of the LBMPC using the approach in \cite{limon2008}.  Moreover, the safety and stability that we have proved for the closed-loop system under LBMPC are actually robust results because they imply that the states remain within bounded constraints even under disturbances, provided the modeling error in (\ref{eqn:linear}) follows the prescribed bound and the invariant set $\Omega$ can be computed.
\end{remark}

Next, we discuss additional types of robustness provided by LBMPC.  First, we show that the value function $V_n(x_n)$ of LBMPC (\ref{eqn:glmpc}) is continuous, and this property can be used for establishing certain other types of robustness of an MPC controller \cite{grimmetal2004,magni2007,rawlings2009model,limon2009}. 

\begin{lemma}
\label{lemma:cont}
Let $\mathcal{X}_F = \{x_n : \exists \mathcal{M}_n\}$ be the feasible region of the LBMPC (\ref{eqn:glmpc}).  If $\psi_n, \mathcal{O}_n$ are continuous, then $V_n(x_n)$ is continuous on $\text{int}(\mathcal{X}_F)$.
\end{lemma}

\begin{proof}
We define a cost function $\tilde{\psi}_n$ and constraint function $\phi$ such that the LBMPC (\ref{eqn:glmpc}) can be rewritten as
\begin{equation}
\begin{aligned}
\textstyle\min_{c,\theta} \ & \tilde{\psi}_n(\theta,x_n,c_n, \ldots,c_{n+N-1}) \\
\text{s.t. } \ & (c,\theta) \in \phi(x_n).
\end{aligned}
\end{equation}
The proof proceeds by showing that both the objective $\tilde{\psi}_n$ and constraint $\phi$ are continuous.  Under such continuity, we get continuity of the value function by the Berge maximum theorem \cite{berge1963} (or equivalently by Theorem C.34 of \cite{rawlings2009model}).

Because the constraints (\ref{eqn:ic}) and (\ref{eqn:lc}) in LBMPC are linear, the constraint $\phi$ is continuous \cite{grimmetal2004}.  Continuity of $\tilde{\psi}_n$ follows by noting that it is the composition of continuous functions --- specifically (\ref{eqn:glmpc}), (\ref{eqn:ic}), and (\ref{eqn:nle}) --- is also a continuous function \cite{rudin1964}.
\end{proof}

\begin{remark}
This result is surprising because a non-convex (and hence nonlinear) MPC problem generally has a discontinuous value function (cf. \cite{grimmetal2004}).  LBMPC is non-convex when $\mathcal{O}_n$ is nonlinear (or $\psi_n$ is non-convex), and the reason that we have a continuous value function is that our active constraints are linear equality constraints or polytopes.  In practice, this result requires being able to numerically compute a global minimum, and this can only be efficiently done for convex optimization problems.
\end{remark}

\begin{remark}
The proof of this result suggests another benefit of LBMPC: The fact that the constraints are linear means that suboptimal solutions can be computed by solving a linear (and hence convex) feasibility problem, even when the LBMPC problem is nonlinear.  This enables more precise tradeoffs between computation and solution accuracy, as compared to conventional forms of nonlinear MPC.
\end{remark}

Next, we prove that LBMPC is robust because its worst case behavior is an increasing function of modeling error.  This type of robustness if formalized by the following definition.  

\begin{definition}[{Grimm, et al. \cite{grimmetal2004}}]
A system is robustly asymptotically stable (RAS) about $\overline{x}_s$ if there exists a type-$\mathcal{KL}$ function $\beta$ and for each $\epsilon > 0$ there exists $\delta > 0$, such that for all $d_n$ satisfying $\max_n\|d_n\| < \delta$ it holds that $x_n \in \mathcal{X}$ and $\|x_n-\overline{x}_s\| \leq \beta(\|x_0 - \overline{x}_s\|,n) + \epsilon$ for all $n \geq 0$.
\end{definition}

\begin{remark}
The intuition is that if a controller for the approximate system (\ref{eqn:linear}) with no disturbance converges to $\overline{x}_s$, then the same controller applied to the approximate system (\ref{eqn:linear}) with bounded disturbance (note that this also includes the true system (\ref{eqn:smodel})) asymptotically remains within a bounded distance from $\overline{x}_s$.  
\end{remark}

We can now prove when LBMPC is RAS.  The key intuitive points are that linear MPC (i.e, LBMPC with an identically zero oracle: $\overline{\mathcal{O}}_n \equiv 0$) needs to be provably convergent for the approximate model with no disturbance, and the oracle for LBMPC needs to be bounded.

\begin{theorem}
\label{theorem:ras}
Assume (a) $\Omega$ has the properties defined in Sect. \ref{sect:fspt}; (b) $\mathcal{M}_0$ is feasible for LBMPC (\ref{eqn:glmpc}) with $x_0$; (c) the cost function $\psi_n$ is time-invariant, continuous, and strictly convex, and (d) there exists a continuous Lyapunov function $W(x)$ for the approximate system (\ref{eqn:linear}) with no disturbance, when using the control law of linear MPC (i.e, LBMPC with $\overline{\mathcal{O}}_n \equiv 0$).  Under these conditions, the control law of LBMPC is RAS with respect to the disturbance $d_n$ in (\ref{eqn:linear}), whenever the oracle $\mathcal{O}_n$ is a continuous function satisfying $\max_{n,\mathcal{X}\times\mathcal{U}}\|\mathcal{O}_n\| \leq \delta$.  Note that this $\delta$ is the same one as from the definition of RAS.
\end{theorem}

\begin{proof}
Let $\overline{\mathcal{M}}_n^*$ be the minimizer for linear MPC, and note that it is unique because $\psi_n$ is assumed to be strictly convex.  Similarly, let $\mathcal{M}_n^*$ be a minimizer for LBMPC.  Now consider the state-dependent disturbance
\begin{equation}
\label{eqn:dist}
e_n = B(\check{u}_n[\mathcal{M}_n^*]-\check{u}_n[\overline{\mathcal{M}}_n^*]) + d_n,
\end{equation}
for the approximate system (\ref{eqn:linear}).  By construction, it holds that $x_{n+1}[\mathcal{M}_n^*] = \overline{x}_{n+1}[\overline{\mathcal{M}}_n^*] + e_n$.

Proposition 8 of \cite{grimmetal2004} and Theorem \ref {theorem:robustmpc} imply that given $\epsilon > 0$, there exists $\delta_1 > 0$ such that for all $e_n$ satisfying $\max_n\|e_n\| < \delta_1$ it holds that $x_n \in \mathcal{X}$ and $\|x_n-\overline{x}_s\| \leq \beta(\|x_0 - \overline{x}_s\|,n) + \epsilon$ for all $n \geq 0$.  What remains to be checked is whether there exists $\delta$ such that $\max_n\|e_n\| < \delta_1$ for the $e_n$ defined in (\ref{eqn:dist}).

The same argument as used in Lemma \ref{lemma:cont} coupled with the strict convexity of the linear MPC gives that $\mathcal{M}_n^*$ is continuous, with respect to $\mathcal{O}_n$, when $\mathcal{O}_n \equiv 0$.  (Recall that the minimizer at this point is $\overline{\mathcal{M}}_n^*$.)  Because of this continuity, this means that there exists $\delta_2 > 0$ such that $\|\check{u}_n[\mathcal{M}_n^*]-\check{u}_n[\overline{\mathcal{M}}_n^*]\| \leq \delta_1/(2\|B\|)$, whenever the oracle lies in the set $\{\mathcal{O}_n : \|\mathcal{O}_n\| < \delta_2\}$.  Taking $\delta = \min\{\delta_1/2,\delta_2\}$ gives the result.
\end{proof}

\begin{remark}
Condition (a) is satisfied if the set $\Omega$ can be computed; it cannot be computed in some situations because it is possible to have $\Omega = \emptyset$.  Conditions (b) and (c) are easy to check.  As we will show in Sect. \ref{section:tracking}, certain systems have easy sufficient conditions for checking the Lyapunov conditions in (d).
\end{remark}

\subsubsection{Example: Tracking in Linearized Systems}
\label{section:tracking}
Here, we show that the Lyapunov condition in Theorem \ref{theorem:ras} can be easily checked when the cost function is quadratic and the approximate model is linear with bounds on its uncertainty.  Suppose we use the quadratic cost defined in \cite{limon2008}
\begin{multline}
\label{eqn:lincost}
\psi_n = \|\tilde{x}_{n+N}-\Lambda\theta\|_P^2 + \|\overline{x}_s - \Lambda\theta\|_T^2 \\
+ \textstyle \sum_{i=0}^{N-1}\|\tilde{x}_{n+i}-\Lambda\theta\|_Q^2 + \|\check{u}_{n+i}-\Psi\theta\|_R^2,
\end{multline}
where $P,Q,R,T$ are positive definite matrices, to track to the point $\overline{x}_s \in \{\Lambda\theta : \Lambda\theta \in \mathcal{X}\}$.  Then, the Lyapunov condition required for Theorem \ref{theorem:ras} holds.

\begin{proposition}
\label{proposition:lincost}
For linear MPC with cost (\ref{eqn:lincost}) where $\overline{x}_s \in \{\Lambda\theta : \Lambda\theta \in \mathcal{X}\}$ is kept fixed, if $(A+BK)$ is Schur stable and $P$ solves the discrete-time Lyapunov equation $(A+BK)'P(A+BK)-P = -(Q+K'RK)$; then there exists a continuous Lyapunov function $W$ for the equilibrium point $\overline{x}_s$ of the approximate model (\ref{eqn:linear}) with no disturbances.
\end{proposition}

\begin{proof}
First note that because we consider the linear MPC case, we have by definition $\tilde{x} = \overline{x}$.

Results from converse Lyapunov theory \cite{jiang2002} indicate that the result is true if the following two conditions hold.  The first is local uniform stability, meaning that for every $\epsilon > 0$, there exists some $\delta > 0$ such that $\|\overline{x}_0 - \overline{x}_s\| < \delta$ implies that $\|\overline{x}_n - \overline{x}_s\| < \epsilon$ for all $n \geq 0$.  The second is that $\lim_{n\rightarrow\infty}\|\overline{x}_n-\overline{x}_s\|=0$ for all feasible points $\overline{x}_0 \in \mathcal{X}_F$.

The second condition was shown in Theorem 1 of \cite{limon2008}, and so we only need to check the first condition.  We begin by noting that since $Q,T$ are positive definite matrices, there exists a positive definite matrix $S$ such that $S < Q$ and $S < T$.  Next, observe that $\|\tilde{x}_{n} - \overline{x}_s\|_S^2 \leq \|\tilde{x}_n-\Lambda\theta\|_Q^2 + \|\overline{x}_s - \Lambda\theta\|_T^2 \leq \psi_n$.  Minimizing the both sides of the inequality subject to the linear MPC constraints yields $\|\tilde{x}_{n} - \overline{x}_s\|_S^2 \leq \overline{V}(\overline{x}_n)$, where $\overline{V}(\overline{x}_n)$ is the value function of the linear MPC optimization.

Because linear MPC is the special case of LBMPC in which $\mathcal{O}_n \equiv 0$, the result in Lemma \ref{lemma:cont} applies: The value function $\overline{V}(\overline{x}_n)$ is continuous.  Furthermore, the proof of Theorem 1 of \cite{limon2008} shows that the value function is non-increasing (i.e., $\overline{V}(\overline{x}_{n+1})) \leq \overline{V}(\overline{x}_n)$), non-negative (i.e., $\overline{V}(\overline{x}_n) \geq 0$), and zero-valued only at the equilibrium point (i.e., $\overline{V}(\overline{x}_s) = 0$).  Because of the continuity of the value function, given $\overline{\epsilon} > 0$ there exists $\overline{\delta} > 0$, such that $\overline{V}(\overline{x}_0) < \overline{\epsilon}$ whenever $\|\overline{x}_0 - \overline{x}_s\| < \overline{\delta}$.  The local uniform stability condition holds by noting that $\|\tilde{x}_{n} - \overline{x}_s\|_S^2 \leq \overline{V}(\overline{x}_n) \leq \overline{V}(\overline{x}_0) = \overline{\epsilon}$, and this proves the result.
\end{proof}

\begin{remark}
The result does not immediately follow from \cite{limon2008}, because the value function of the linear MPC is not a Lyapunov function in this situation.  In particular, the value function is non-increasing, but it is not strictly decreasing.
\end{remark}

\section{The Oracle}

In theoretical computer science, oracles are black boxes that take in inputs and give answers.  An important class of arguments known as relativizing proofs utilize oracles in order to prove results in complexity theory and computability theory.  These proofs proceed by endowing the oracle with certain generic properties and then studying the resulting consequences.

We have named the functions $\mathcal{O}_n$ oracles in reference to those in computer science.  Our reason is that we proved robustness and stability properties of LBMPC by only assuming generic properties, such as continuity or boundedness, on the function $\mathcal{O}_n$.  These functions are arbitrary, which can include worst case behavior, for the purpose of the theorems in the previous section.

Whereas the previous section considered the oracles as abstract objects, here we discuss and study specific forms that the oracle can take.  In particular, we can design $\mathcal{O}_n$ to be a statistical tool that identifies better system models.  This leads to two natural questions: First, what are examples of statistical methods that can be used to construct an oracle for LBMPC?  Secondly, when does the control law of LBMPC converge to the control law of MPC that knows the true model?


This section begins by defining two general classes of statistical tools that can be used to design the oracle $\mathcal{O}_n$.  For concreteness, we provide a few examples of methods that belong to these two classes.  The section concludes by addressing the second question above.  Because our control law is the minimizer of an optimization problem, the key technical issue that we discuss is sufficient conditions that ensure convergence of the minimizers of a sequence of optimization problems to the minimizer of a limiting optimization problem.

\subsection{Parametric Oracles}

A parametric oracle is a continuous function $\mathcal{O}_n(x,u) = \chi(x,u;\lambda_n)$ that is parameterized by a set of coefficients $\lambda_n \in \mathcal{T} \subseteq \mathbb{R}^L$, where $\mathcal{T}$ is a set.  This class of learning is often used in adaptive control \cite{sastry1989,astrom1995}.  In the most general case, the function $\chi$ is nonlinear in all its arguments, and it is customary to use a least-squares cost function with input and trajectory data to estimate the parameters
\begin{equation}
\label{eqn:cost}
\hat{\lambda}_n = \arg \textstyle \min_{\lambda \in \mathcal{T}} \textstyle\sum_{j=0}^n(Y_j - \chi(x_j,u_j;\lambda))^2,
\end{equation}
where $Y_i = x_{i+1} - (Ax_i+Bu_i)$.  This can be difficult to compute in real-time because it is generally a nonlinear optimization problem.

\begin{example}
It is common in biochemical networks to have nonlinear terms in the dynamics such as
\begin{equation}
\mathcal{O}_n(x,u) = \lambda_{n,1}\Bigg(\frac{x_1^{\lambda_{n,2}}}{x_1^{\lambda_{n,2}} + \lambda_{n,3}}\Bigg)\Bigg(\frac{\lambda_{n,4}}{u_1^{\lambda_{n,5}} + \lambda_{n,4}}\Bigg),
\end{equation}
where $\lambda_n \in \mathcal{T} \subset \mathbb{R}^5$ are the unknown coefficients in this example.  Such terms are often called Hill equation type reactions \cite{aswani2009b}.
\end{example}

An important subclass of parametric oracles are those that are linear in the coefficients: $\mathcal{O}_n(x,u) = \sum_{i=1}^L\lambda_{n,i}\chi_i(x,u)$, where $\chi_i \in \mathbb{R}^p$ for $i = 1,\ldots,L$ are a set of (possibly nonlinear) functions.  The reason for the importance of this subclass is that the least-squares procedure (\ref{eqn:cost}) is convex in this situation, even when the functions $\chi_i$ are nonlinear.  This greatly simplifies the computation required to solve the least-squares problem (\ref{eqn:cost}) that gives the unknown coefficients $\lambda_n$.

\begin{example}
One special case of linear parametric oracles is when the $\chi_i$ are linear functions.  Here, the oracle can be written as $\mathcal{O}_m(x,u) = F_{\lambda_m}x + G_{\lambda_m}u$, where $F_{\lambda_m},G_{\lambda_m}$ are matrices whose entries are parameters.  The intuition is that this oracle allows for corrections to the values in the $A,B$ matrices of the nominal model; it was used in conjunction with LBMPC on a quadrotor helicopter testbed \cite{aswani_quad_2011,bouffard2011}, in which LBMPC enabled high-performance flight.
\end{example}

\subsection{Nonparametric Oracles}

Nonparametric regression refers to techniques that estimate a function $g(x,u)$ of input variables such as $x,u$, without making \textit{a priori} assumptions about the mathematical form or structure of the function $g$.  This class of techniques is interesting because it allows us to integrate non-traditional forms of adaptation and ``learning'' into LBMPC.  And because LBMPC robustly maintains feasibility and constraint satisfaction as long as $\Omega$ can be computed, we can design or choose the nonparametric regression method without having to worry about stability properties.  This is a specific instantiation of the separation between robustness and performance in LBMPC.

\begin{example}
Neural networks are a classic example of a nonparametric method that has been used in adaptive control \cite{narendra1990,rovithakis1994,anderson2007}, and they can also be used with LBMPC.  There are many particular forms of neural networks, and one specific type is a feedforward neural network with a hidden layer of $k_n$ neurons; it is given by
\begin{equation}
\mathcal{O}_n(x,u) = c_0 + \textstyle\sum_{i = 1}^{k_n}c_i\sigma(a_i'[x'\ u']'+b_i),
\end{equation}
where $a_i \in \mathbb{R}^{p+m}$ and $b_i,c_0,c_i \in \mathbb{R}$ for all $i \in \{1,\ldots,k\}$ are coefficients, and $\sigma(x) = 1/(1 + e^{-x}) : \mathbb{R} \rightarrow [0,1]$ is a sigmoid function \cite{nonparametric2002}.  Note that this is considered a nonparametric method because it does not generally converge unless $k_n \rightarrow \infty$ as $n \rightarrow \infty$.
\end{example}

Designing a nonparametric oracle for LBMPC is challenging because the tool should ideally be an estimator that is bounded to ensure robustness of LBMPC and differentiable to allow for its use with numerical optimization algorithms.  Local linear estimators \cite{ruppertwand1994,aswani2010} are not guaranteed to be bounded, and their extensions that remain bounded are generally non-differentiable \cite{fiacco1976}.  On the other hand, neural networks can be designed to remain bounded and differentiable, but they can have technical difficulties related to the estimation of its coefficients \cite{vapnik1999}.

\subsubsection{Example: $L2$-Regularized Nadaraya-Watson Estimator}

The Nadaraya-Watson (NW) estimator \cite{muller1987,ruppertwand1994}, which can be intuitively thought of as the interpolation of non-uniformly sampled data points by a suitably normalized convolution kernel, is promising because it ensures boundedness.  Our approach to designing a nonparametric estimator for LBMPC is to modify the NW estimator by adding regularization that deterministically ensures boundedness.  Thus, it serves the same purpose as \textit{trimming} \cite{bickel1982}; but the benefit of our approach is that it also deterministically ensures differentiability of the estimator.  To our knowledge, this modification of NW has not been previously considered in the literature.

Define $h_n,\lambda_n \in \mathbb{R}_+$ to be two non-negative parameters; except when we wish to emphasize their temporal dependence, we will drop the subscript $n$ to match the convention of the statistics literature.  Let $X_i = [x_i'\  u_i']'$, $Y_i = x_{i+1} - (Ax_i+Bu_i)$, and $\Xi_i = \|\xi - x_i\|^2/h^2$, where $X_i \in \mathbb{R}^{p+m}$ and $Y_i \in \mathbb{R}^p$ are data and $\xi = [x' \ u']'$ are free variables.  We define any function $\kappa : \mathbb{R} \rightarrow \mathbb{R}_+$ to be a kernel function if it has (a) finite support (i.e., $\kappa(\nu) = 0$ for $|\nu| \geq 1$), (b) even symmetry $\kappa(\nu) = \kappa(-\nu)$, (c) positive values $\kappa(\nu) > 0$ for $|\nu| < 1$, (d) differentiability (denoted by $d\kappa$), and (e) nonincreasing values of $\kappa(\nu)$ over $\nu \geq 0$.  The $L2$-regularized NW (L2NW) estimator is defined as
\begin{equation}
\label{eqn:rnw}
\mathcal{O}_n(x,u) = \frac{\sum_i Y_i \kappa(\Xi_i)}{\lambda + \sum_i \kappa(\Xi_i)},
\end{equation}
where $\lambda \in \mathbb{R}_+$.  If $\lambda = 0$, then (\ref{eqn:rnw}) is simply the NW estimator.  The $\lambda$ term acts to regularize the problem and ensures differentiability.  

There are two alternative characterizations of (\ref{eqn:rnw}).  The first is as the unique minimizer of the parametrized, strictly convex optimization problem $\mathcal{O}_n(x,u) = \arg \min_{\gamma} L(x,u,X_i,Y_i,\gamma)$ for 
\begin{equation}
\label{eqn:onw}
L(x,u,X_i,Y_i,\gamma) = \textstyle\sum_i \kappa(\Xi_i)(Y_i - \gamma)^2 + \lambda\gamma^2.
\end{equation}
Viewed in this way, the $\lambda$ term represents a Tikhonov (or $L2$) regularization \cite{tikhonov1977,hoerl1970}.  The second characterization is as the mean with weights $\{\lambda,\kappa(\Xi_1),\ldots,\kappa(\Xi_n)\}$ for points $\{0, Y_1, \ldots, Y_n\}$, and it is useful for showing the second part of the following theorem about the deterministic properties of the L2NW estimator.

\begin{theorem}
\label{theorem:dnw}
If $0 \in \mathcal{W}$, $\kappa(\cdot)$ is a kernel function, and $\lambda > 0$; then (a) the L2NW estimator $\mathcal{O}_n(x,u)$ as defined in (\ref{eqn:rnw}) is differentiable, and (b) $\mathcal{O}_n(x,u) \in \mathcal{W}$.
\end{theorem}

\begin{proof}
To prove (a), note that the estimate $\mathcal{O}_n(x,u)$ is the value of $\gamma$ that solves $\frac{dL}{d\gamma}(x,u,X_i,Y_i,\gamma) = 0$, where $L(\cdot)$ is from (\ref{eqn:onw}).  Because $\lambda + \sum_i \kappa(\Xi_i) > 0$, the hypothesis of the implicit function theorem is satisfied, and result directly follows from the implicit function theorem.

Part (b) is shown by noting that the assumptions imply that $0,Y_i \in \mathcal{W}$.  If the weights of a weighted mean are positive and have a nonzero sum, then the weighted mean can be written as a convex combination of points.  This is our situation, and so the result follow from the weighted mean characterization of (\ref{eqn:rnw}). 
\end{proof}

\begin{remark}
This shows that L2NW is deterministically bounded and differentiable, which is needed for robustness and numerical optimization, respectively.  We can compute the gradient of L2NW using standard calculus, and its $jk$-th component is given by (\ref{eqn:gnw}) for fixed $X_i,Y_i$.
\begin{figure*}[!t]
\normalsize
\begin{equation}
\label{eqn:gnw}
\frac{\partial\mathcal{O}_n}{\partial \xi_k}\Big(x,u\Big) = \frac{\{\textstyle\sum_i [Y_i]_j\cdot d\kappa(\Xi_i)\cdot \Xi_i\cdot [\xi - X_i]_k\}\{\lambda + \textstyle\sum_i \kappa(\Xi_i)\} - \{\textstyle\sum_i [Y_i]_j\kappa(\Xi_i)\}\{\textstyle\sum_i d\kappa(\Xi_i)\cdot \Xi_i\cdot [\xi - X_i]_k\}}{h^2\{\lambda + \sum_i \kappa(\Xi_i)\}^2/2}.
\end{equation}
\hrulefill
\vspace*{4pt}
\end{figure*}
\end{remark}

There are few notes regarding numerical computation of L2NW.  First, picking the parameters $\lambda,h$ in a data-driven manner \cite{efron1979,shao1993} is too slow for real-time implementation, and so we suggest rules of thumb: Deterministic regularity is provided by Theorem \ref{theorem:dnw} for any positive $\lambda$ (e.g., 1e-3), and we conjecture using $h_n = O(n^{-1/(p+m)})$ because random samples cover $\mathcal{X}\times\mathcal{U} \subseteq \mathbb{R}^{p+m}$ at this rate.  Second, computational savings are possible through careful software coding, because if $h$ is small, then most terms in the summations of (\ref{eqn:onw}) and (\ref{eqn:gnw}) will be zero because of the finite support of $\kappa(\cdot)$.

\subsection{Stochastic Epi-convergence}

It remains to be shown that if $\mathcal{O}_n(x,u)$ stochastically converges to the true model $g(x,u)$, then the control law of the LBMPC scheme will stochastically converge to that of an MPC that knows the true model.  The main technical problem occurs because $\mathcal{O}_n$ is time-varying, and so the control law is given by the minimizer of an LBMPC optimization problem that is different at each point in time $n$.  This presents a problem because pointwise convergence of $\mathcal{O}_n$ to $g$ is generally insufficient to prove convergence of the minimizers of a sequence of optimization problems to the minimizer of a limiting optimization problem \cite{rockafellarwets1998,vogellachout2003a}.

A related notion called epi-convergence is sufficient for showing convergence of the control law.  Define the \textit{epigraph} of $f_n(\cdot,\omega)$ to be the set of all points lying on or above the function, and denote it as $\text{Epi }f_n(\cdot,\omega) = \{(x,\mu) : \mu \geq f_n(x,\omega)\}$.  To prove convergence of the sequence of minimizers, we must show that the epigraph of the cost function (and constraints) of the sequence of optimizations converges in probability to the epigraph of the cost function (and constraints) in the limiting optimization problem.  This notion is called epi-convergence, and we denote it as $f_n \xrightarrow[\mathcal{X}]{l-prob.} f_0$.

For simplicity, we will assume in this section that the cost function is time-invariant (i.e., $\psi_n \equiv \psi_0$).  It is enough to cite the relevant results for our purposes, but the interested reader can refer to \cite{rockafellarwets1998,vogellachout2003a} for details. 
\begin{theorem}[{Theorem 4.3 \cite{vogellachout2003a}}]
Let $\tilde{\psi}_n$ and $\phi$ be as defined in Lemma \ref{lemma:cont}, and define $\tilde{\psi}_0$ to be the composition of (\ref{eqn:glmpc}) with both (\ref{eqn:ic}) and $x_{n+i+1}(x_{n+i},u_{n+i}) = Ax_{n+i} + Bu_{n+i} + g(x_{n+i},u_{n+i})$.  If $\tilde{\psi}_n \xrightarrow[\phi(x_n)]{l-prob.} \tilde{\psi}_0$ for all $\{x_n : \phi(x_n) \neq \emptyset\}$, then the set of minimizers converges
\begin{multline}
\arg\min\{\tilde{\psi}_n | (c,\theta) \in \phi(x_n)\} \\ \xrightarrow{p} \arg\min\{\tilde{\psi}_0 | (c,\theta) \in \phi(x_n)\}.
\end{multline}
\end{theorem}

\begin{remark}
The intuition is that if the cost function $\psi_n$ composed with the oracle $\mathcal{O}_n(x,u)$ converges in the appropriate manner to $\psi_0$ composed with the true dynamics $g(x,u)$; then we get convergence of the minimizers of LBMPC to those of the MPC with true model, and the control law (\ref{eqn:feedbackmpc}) converges.  This theorem can be used to prove convergence of the LBMPC control law.
\end{remark}


\subsection{Epi-convergence for Parametric Oracles}

Sufficient excitation (SE) is an important concept in system identification, and it intuitively means that the control inputs and state trajectory of the system are such that all modes of the system are activated.  In general, it is hard to design a control scheme that ensures this \textit{a priori}, which is a key aim of reinforcement learning \cite{barto2004}.  However, LBMPC provides a framework in which SE may be able to be designed.  Because we have a nominal model, we can in principle design a reference trajectory that sufficiently explores the state-input space $\mathcal{X}\times\mathcal{U}$.


Though designing a controller that ensures SE can be difficult, checking \textit{a posteriori} whether a system has SE is straightforward \cite{ljung1987,aswani2009,aswani2010}.  In this section, we assume SE and leave open the problem of how to design reference trajectories for LBMPC that guarantee SE.  This is not problematic from the standpoint of stability and robsutness, because LBMPC provides these properties, even without SE, whenever the conditions in Sect. \ref{section:lbmpc} hold.  We have convergence of the control law assuming SE, statistical regularity, and that the oracle can correctly model $g(x,u)$.  The proof of the following theorem can be found in \cite{aswani2011_techrep}

\begin{theorem}
Suppose there exists $\lambda_0 \in \mathcal{T}$ such that $g(x,u) = \chi(x,u;\lambda_0)$.  If the system has SE \cite{lai1979,jennrich1969,malinvaud1970}, then the control law of the LBMPC with oracle (\ref{eqn:cost}) converges in probability to the control law of an MPC that knows the true model (i.e., $u_n[\mathcal{M}_n^*] \xrightarrow{p} u_0[\mathcal{M}_0^*]$).
\end{theorem}

\subsection{Epi-convergence for Nonparametric Oracles}

For a nonlinear system, SE is usually defined using ergodicity or mixing, but this is hard to verify in general.  Instead, we define SE as a finite sample cover (FSC) of $\mathcal{X}$.  Let $\mathcal{B}_h(x) = \{y : \|x-y\| \leq h\}$ be a ball centered at $x$ with radius $h$, then a FSC of $\mathcal{X}$ is a set $\mathcal{S}_h = \bigcup_i \mathcal{B}_{h/2}(X_i)$ that satisfies $\mathcal{X} \subseteq \mathcal{S}_h$.  The intuition is that $\{X_i\}$ sample $\mathcal{X}$ with average, inter-sample distance less than $h/2$.  

Our first result considers a generic nonparametric oracle with uniform pointwise convergence.  Such uniform convergence implicitly implies SE in the form of a FSC with asymptotically decreasing radius $h$ \cite{zakai2008}, though we make this explicit in our statement of the result.  A proof can be found in \cite{aswani2011_techrep}.
\begin{theorem}
Let $h_n$ be some sequence such that $h_n \rightarrow~0$.  If $\mathcal{S}_{h_n}$ is a FSC of $\mathcal{X} \times \mathcal{U}$ and
\begin{equation}
\sup_{\mathcal{X}\times\mathcal{U}}\|\mathcal{O}_n(x,u)-g(x,u)\| = O_p(r_n),
\end{equation}
with $r_n \rightarrow 0$; then the control law of LBMPC with $\mathcal{O}_n(x,u)$ converges in probability to the control law of an MPC that knows the true model (i.e., $u_n[\mathcal{M}_n^*] \xrightarrow{p} u_0[\mathcal{M}_0^*]$).
\end{theorem}

\begin{remark}
Our reason for presenting this result is that this theorem may be useful for proving convergence of the control law when using types of nonparametric regression that are more complex than L2NW.  However, we stress that this is a sufficient condition, and so it may be possible for nonparametric tools that do not meet this condition to generate such stochastic convergence of the controller.
\end{remark}

Assuming SE in the form of a FSC with asymptotically decreasing radius $h$, we can show that the control law of LBMPC that uses L2NW converges to that of an MPC that knows the true dynamics.  Because the proofs \cite{aswani2011_techrep} rely upon theory from probability and statistics, we simply summarize the main result.
\begin{theorem}
Let $h_n$ be some sequence such that $h_n \rightarrow 0$.  If $\mathcal{S}_{h_n}$ is a FSC of $\mathcal{X} \times \mathcal{U}$, $\lambda = O(h_n)$, and $g(x,u)$ is Lipschitz continuous; then the control law of LBMPC with L2NW converges in probability to the control law of an MPC that knows the true model (i.e., $u_n[\mathcal{M}_n^*] \xrightarrow{p} u_0[\mathcal{M}_0^*]$).
\end{theorem}


\section{Experimental and Numerical Results}

In this section, we briefly discuss applications in which LBMPC has been experimentally applied to different testbeds.  The section concludes with numerical simulations that display some of the features of LBMPC.

\subsection{Energy-efficient Building Automation}

We have implemented LBMPC on two testbeds that were built on the Berkeley campus for the purpose of study energy-efficient control of heating, ventilation, and air-conditioning (HVAC) equipment.  The first testbed \cite{aswani_proc}, which is named the Berkeley Retrofitted and Inexpensive HVAC Testbed for Energy Efficiency (BRITE), is a single-room that uses HVAC equipment that is commonly found in homes.  LBMPC was able to generate up to 30\% energy savings on warm days and up to 70\% energy savings on cooler days, as compared to the existing control of the thermostat within the room.  It achieved this by using semiparametric regression to be able to estimate, using only temperature measurements from the thermostat, the heating load from exogenous sources like occupants, equipment, and solar heating.  The LBMPC used this estimated heating load as its form of learning, and was able to adjust the control action of the HVAC based on this in order to achieve large energy savings.

The second testbed \cite{aswani2012nmpc}, which is named BRITE in Sutardja Dai Hall (BRITE-S), is a seven floor office building that is used in multiple ways.  The building has offices, classrooms, an auditorium, laboratory space, a kitchen, and a coffee shop with dining area.  Using a variant of LBMPC for hybrid systems with controlled switches, we were able to achieve an average of 1.5MWh of energy savings per day.  For reference, eight days of energy savings is enough to power an average American home for one year.  Again, we used semiparametric regression to be able to estimate, using only temperature measurements from the building, the heating load from exogenous sources like occupants, equipment, and solar heating.  The LBMPC used this estimated heating load along with additional estimates of unmodeled actuator dynamics, as its form of learning, in order to adjust its supervisory control action.

\subsection{High Performance Quadrotor Helicopter Flight}

We have also used LBMPC in order to achieve high performance flight for semi-autonomous systems such as a quadrotor helicopter, which is a non-traditional helicopter with four propellers that enable improved steady-state stability properties \cite{HWT08}.  In our experiments with LBMPC on this quadrotor testbed \cite{aswani_quad_2011,bouffard2011}, the learning was implemented using an extended Kalman filter (EKF) that provided corrections to the coefficients in the $A,B$ matrices.  This makes it similar to LPV-MPC, which performs linear MPC using a successive series of linearizations of a nonlinear model; in our case, we used the learning provided by the EKF to in effect perform such linearizations.

Various experiments that we conducted showed that LBMPC improved performance and provided robustness.  Amongst the experiments we performed were those that (a) showed improved step responses with lower amounts of overshoot and settling time as compared to linear MPC, and (b) displayed the ability of the LBMPC controller to overcome a phenomenon known as the ground effect that typically makes flight paths close to the ground difficult to perform.  Furthermore, the LBMPC displayed robustness by preventing crashes into the ground during experiments in which the EKF was purposely made unstable in order to mis-learn.  The improved performance and learning generalization possible with the type of adaptation and learning within LBMPC was demonstrated with an integrated experiment in which the quadrotor helicopter caught ping-pong balls that were thrown to it by a human.

\subsection{Example: Moore-Greitzer Compressor Model}

Here, we present a simulation of LBMPC on a nonlinear system for illustrative purposes.  The compression system of a jet engine can exhibit two types of instability: rotating stall and surge \cite{mooregreitzer1986,epstein1989,kristic1995}.  Rotating stall is a rotating region of reduced air flow, and it degrades the performance of the engine.  Surge is an oscillation of air flow that can damage the engine.  Historically, these instabilities were prevented by operating the engine conservatively.  But better performance is possible through active control schemes \cite{epstein1989,kristic1995}.

The Moore-Greitzer model is an ODE model that describes the compressor and predicts surge instability
\begin{equation}
\label{eqn:mg}
\begin{aligned}
\dot{\Phi} &= -\Psi + \Psi_c + 1 + 3\Phi/2-\Phi^3/2\\
\dot{\Psi} &= (\Phi + 1 - r \sqrt{\Psi})/\beta^2,
\end{aligned}
\end{equation}
where $\Phi$ is mass flow, $\Psi$ is pressure rise, $\beta > 0$ is a constant, and $r$ is the throttle opening.  We assume $r$ is controlled by a second order actuator with transfer function $r(s) = \frac{w_n^2}{s^2 + 2\zeta w_n s + w_n^2}u(s)$, where $\zeta$ is the damping coefficient, $w_n^2$ is the resonant frequency, and $u$ is the input.

We conducted simulations of this system with the parameters $\beta = 1$, $\Psi_c = 0$, $\zeta = 1/\sqrt{2}$, and $w_n = \sqrt{1000}$.  We chose state constraints $0 \leq \Phi \leq 1$ and $1.1875 \leq \Psi \leq 2.1875$, actuator constraints $0.1547 \leq r \leq 2.1547$ and $-20 \leq \dot{r} \leq 20$, and input constraints $0.1547 \leq u \leq 2.1547$.  For the controller design, we took the approximate model with state $\delta x = [\delta\Phi \ \delta\Psi \ \delta r \ \delta \dot{r}]'$ to be the exact discretization (with sampling time $T = 0.01$) of the linearization of (\ref{eqn:mg}) about the equilibrium $x_0 = [\Phi_0 \ \Psi_0 \ r_0 \ \dot{r}_0]' = [0.5000 \ 1.6875 \ 1.1547 \ 0]'$; the control is $u_n = \delta u_n + u_0$, where $u_0 \equiv r_0$.  The linearization and approximate model are unstable, and so we picked a nominal feedback matrix $K = [-3.0741 \ 2.0957 \ 0.1195 \ -0.0090]$ that stabilizes the system by ensuring that the poles of the closed-loop system $x_{n+1} = (A+BK)x_n$ were placed at $\{0.75,0.78,0.98,0.99\}$.  These particular poles were chosen because they are close to the poles of the open-loop system, while still being stable.

For the purpose of computing the invariant set $\Omega$, we used the algorithm in \cite{kolmanovsky1998}.  This algorithm uses the modeling error set $\mathcal{W}$ as one of its inputs.  This set $\mathcal{W}$ was chosen to be a hypercube that encompasses both a bound on the linearization error, derived using the Taylor remainder theorem applied to the true nonlinear model, along with a small amount of subjectively-chosen ``safety margin'' to provide protection against the effect of numerical errors.

We compared the performance of linear MPC, nonlinear MPC, and LBMPC with L2NW for regulating the system about the operating point $x_0$, by conducting a simulation starting from initial condition $[\Phi_0-0.35 \ \Psi_0-0.40 \ r_0 \ 0]'$.  The horizon was chosen to be $N = 100$, and we used the cost function (\ref{eqn:lincost}), with $Q=\mathbb{I}_4$, $R = 1$, $T = 1e3$, and $P$ that solves the discrete-time Lyapunov equation.  The L2NW used an Epanechnikov kernel (CITE), with parameter values $h=0.5$, $\lambda=\text{1e-3}$ and data measured as the system was controlled by LBMPC.  Also, the L2NW only used three states $X_i = [\Phi_i \ \Psi_i \ u_i]$ to estimate $g(x,u)$; incorporation of such prior knowledge improves estimation by reducing dimensionality.

The significance of this setup is that the assumptions of Theorems \ref{theorem:robustmpc} and \ref{theorem:ras} (via Proposition \ref{proposition:lincost}) are satisfied.  This means that for both linear MPC and LBMPC: (a) constraints and feasibility are robustly maintained despite modeling errors, (b) closed-loop stability is ensured, and (c) control is ISS with respect to modeling error.  In the instances we simulated, the controllers demonstrated these features.  More importantly, this example shows that the conditions of our deterministic theorems can be checked easily for interesting systems such as this.

Simulation results are shown in Fig. \ref{fig:mg_plot}: LBMPC converges faster to the operating point than linear MPC, but requires increased  computation at each step ($0.3s$ for linear MPC vs. $0.9s$ for LBMPC).  Interestingly, LBMPC performs as well as nonlinear MPC, but nonlinear MPC only requires $0.4s$ to compute each step.  However, our point is that LBMPC does not require the control engineer to model nonlinearities, in contrast to nonlinear MPC.  Our code was written in MATLAB and uses the SNOPT solver \cite{snopt2005} for optimization; polytope computations used the Multi-Parametric Toolbox (MPT) \cite{mpt}.

\begin{figure}[t!]
  \begin{center}
      \includegraphics{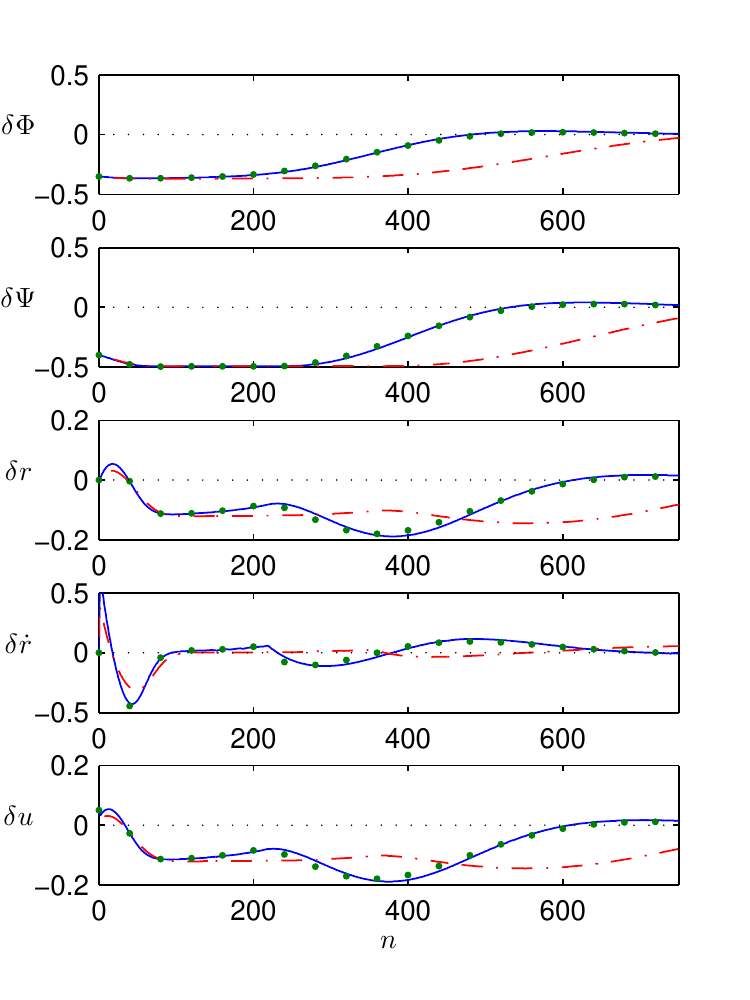}
    \caption{The states and control of LBMPC (solid blue), linear MPC (dashed red), and nonlinear MPC (dotted green) are shown.  LBMPC converges faster than linear MPC.}
    \label{fig:mg_plot}
  \end{center}
\end{figure}

\section{Conclusion}

LBMPC uses a linear model with bounds on its uncertainty to construct invariant sets that provide deterministic guarantees on robustness and safety.  An advantage of LBMPC is that many types of statistical identification tools can be used with it, and we constructed a new nonparametric estimator that has deterministic properties required for use with numerical optimization algorithms while also satisfying conditions required for robustness.  A simulation shows that LBMPC can improve over linear MPC, and experiments on testbeds \cite{aswani_proc,aswani_quad_2011,bouffard2011} show that such improvement translates to real systems.  

Amongst the most interesting directions for future work is the design of better learning methods for use in LBMPC.  Loosely speaking, nonparametric methods work by localizing measurements in order to provide consistent estimates of the function $g(x,u)$ \cite{zakai2008}.  The L2NW estimator maintains \textit{strict locality} in the sense of \cite{zakai2008}, because this property makes it easier to perform theoretical analysis.  However, it is known that learning methods that also incorporate global regularization, such as support vector regression \cite{smola2004,vapnik1999}, can outperform strictly local methods \cite{zakai2008}.  The design of such globally-regularized nonparametric methods which also have theoretical properties favorable for LBMPC is an open problem.

\begin{ack}                               
The authors would like to acknowledge Jerry Ding and Ram Vasudevan for useful discussions about collocation.  This material is based upon work supported by the National Science Foundation under Grant No. 0931843, the Army Research Laboratory under Cooperative Agreement Number W911NF-08-2-0004, the Air Force Office of Scientific Research under Agreement Number FA9550-06-1-0312, and PRET Grant 18796-S2.
\end{ack}

\bibliographystyle{plain}        
\bibliography{safe_mpc}           


\end{document}